\documentclass[12pt,oneside]{amsart}

\usepackage{amssymb,latexsym}

\usepackage{hyperref}
\hypersetup{
  colorlinks   = true, 
  urlcolor     = blue, 
  linkcolor    = blue, 
  citecolor   = red 
}
\usepackage{anysize}
\marginsize{2cm}{2cm}{2cm}{2cm}

\usepackage[pdftex]{graphicx}

\newtheorem{theorem}{Theorem}[section]

\newtheorem{lemma}[theorem]{Lemma}
\newtheorem{claim}[theorem]{Claim}
\newtheorem{conjecture}[theorem]{Conjecture}
\newtheorem{corollary}[theorem]{Corollary}

\theoremstyle{definition}

\theoremstyle{remark}
\newtheorem{remark}[theorem]{Remark}
\newtheorem{question}[theorem]{Question}
\newtheorem{example}[theorem]{Example}

\numberwithin{equation}{section}

\DeclareMathOperator{\vol}{vol}

\DeclareMathOperator{\Hor}{Hor}
\DeclareMathOperator{\Ver}{Ver}

\renewcommand{\epsilon}{\varepsilon}
\renewcommand{\phi}{\varphi}
\renewcommand{\kappa}{\varkappa}

\title{Convex bodies with all characteristics planar}

\author{Roman~Karasev{$^\spadesuit$}}

\address{Roman Karasev, Institute for Information Transmission Problems RAS, Bolshoy Karetny per. 19, Moscow, Russia 127994}

\email{r\_n\_karasev@mail.ru}
\urladdr{http://www.rkarasev.ru/en/}

\author{Anastasiia~Sharipova{$^\clubsuit$}}

\address{Anastasiia Sharipova, Department of Mathematics, Pennsylvania State University, University Park, PA
16802, USA}

\email{independsharik@gmail.com}

\thanks{{$^\clubsuit$}This work is supported by NSF grant DMS-2005444}


\begin{document}

\begin{abstract}
We show that smooth and strongly convex bodies in the symplectic $\mathbb R^{2n}$ for $n>1$ with all characteristics planar, or all outer billiard trajectories planar are affine symplectic images of balls.
\end{abstract}

\maketitle

\section{Introduction}

For a smooth convex body $K\subset\mathbb R^{2n}$ one defines its characteristics as curves on $\partial K$ whose tangent vectors are contained in 
\[
\ker \omega|_{T_z\partial K} = JN_z K,\quad z\in \partial K.
\]
Here $\omega=\sum_{i=1}^n dp_i\wedge dq_i$ is the standard symplectic structure on $\mathbb R^{2n}$, $J$ is the standard complex structure on $\mathbb R^{2n}\cong \mathbb C^n$, $N_zK$ is the line of vectors normal to $\partial K$ at a point $z$. The characteristics of $K$ are the Hamiltonian trajectories on $\partial K$ whenever $\partial K$ is a level set of a Hamiltonian.

In this paper we show that (Lemma~\ref{lemma:same-action}) if all characteristics of $K$ are closed with the same symplectic action $c$ (the integral of $\omega$ by a surface filling the characteristic) then the volume of $K$ equals $c^n/n!$. Moreover, if all characteristics of $K$ are planar (Theorem \ref{theorem:all-planar}) with an additional technical assumption of strong convexity then $K$ is an affine symplectic image of a ball. Those results are inspired by the following natural conjecture (confirmed for $2n=4$ in \cite[Proposition~4.3]{abhs2018} and proved for bodies sufficiently $C^3$-close to balls in \cite[Proposition~4]{abben2019}).

\begin{conjecture}
\label{conjecture:all-same}
If $K\subset\mathbb R^{2n}$ is a smooth convex body with all characteristics closed with the same action then $K$ is a $($possibly non-linear$)$ symplectomorphic image of a ball.
\end{conjecture}

As noted by the referee, it is not clear if the convexity in this conjecture is necessary, but we would like to concentrate on convex bodies in this paper. This conjecture is related to our result since the planarity of all characteristics implies that they are all closed with the same action (see Lemma~\ref{lemma:planar same}). Hence we make stronger assumptions and deduce a stronger consequence, an affine symplectomorphism instead of an arbitrary symplectomorphism.

There is a nice characterization of a ball in terms of a billiard dynamics \cite{sine1976}: A smooth convex body in $\mathbb{R}^n$ with $n>2$ is a ball if and only if all its billiard trajectories are planar. We call a trajectory \emph{planar} if it lies in some affine $2$-dimensional plane.

Another motivation for us is that one can formulate a similar statement in terms of outer (dual) billiard dynamics \cite{tabachnikov1993}.

\begin{conjecture}
\label{conjecture:outer-billiards}
All outer billiard trajectories of a smooth and strictly convex body in dimension $2n\ge 4$ are planar if and only if the body is a ball up to an affine symplectomorphism.
\end{conjecture}

Recall that an outer billiard is a dynamical system where a point $x$ outside of a convex body $K \subset \mathbb R^{2n}$ is mapped to the point $y$ on the tangent line from $x$ to $K$ along the characteristic direction at the tangency point such that the point of tangency is the midpoint between $x$ and $y$. The outer billiard map looks like a discretization of the characteristic flow on $\partial K$, although the connection is not straightforward, see the argument in Section~\ref{section:outer} for example. For a detailed introduction to outer (dual) billiards the reader is referred to \cite{tabachnikov1993,tabachnikov2003}. Another billiard-like characterization of a ball is given in \cite{bialy2018}.

Below we formulate and prove a result about planar characteristics, and deduce from this the statement for outer billiard trajectories.

\begin{theorem}
\label{theorem:all-planar}
Let $n \ge 2$ and $K \subset \mathbb{R}^{2n}$ be a strongly convex body with smooth boundary. Then all characteristics on $\partial K$ are planar if and only if $K$ is a ball up to an affine symplectomorphism.
\end{theorem}

We say that a body is \emph{strongly convex} if it has positive curvature $($positive definite second fundamental form$)$ at every boundary point. Our proof essentially uses the positive curvature (strong convexity) assumption. If we allow $K$ to be unbounded (note that the definition of a \emph{convex body} assumes it is bounded) with no restriction on strong convexity then the theorem may indeed fail. Namely, all characteristics on the boundary of the symplectic cylinder $Z = \{ (q, p) \in \mathbb R^{2n} |\ q_1^2 + p_1^2 = 1\}$ are planar and the cylinder is not symplectomorphic to a ball because of having infinite volume. This justifies the following question.

\begin{question}
Does Theorem \ref{theorem:all-planar} hold true for smooth convex bodies without any assumption on the curvature?
\end{question}

Our proof is mainly built on global invariants of convex bodies in $\mathbb R^{2n}$, the volume, the symplectic capacity, and even the cohomology of their quotient spaces. In particular, we cannot give any information on convex bodies that have planar characteristics only locally. Hence there remains the following question.

\begin{question}
Assume that a smooth convex body $K$ has a planar characteristic $\gamma\subset \partial K$ and all its characteristics sufficiently close to $\gamma$ are also planar. Can one describe the structure of a neighborhood of $\gamma$ in $\partial K$? Can one prove that under the additional positive curvature assumption a neighborhood of $\gamma$ in $\partial K$ is a part of the surface of an ellipsoid?
\end{question}

In the last section of this text we infer the following result about outer billiards from Theorem~\ref{theorem:all-planar}.

\begin{corollary}
\label{corollary:outer-planar}
Let $n \ge 2$ and $K \subset \mathbb{R}^{2n}$ be a strongly convex body with smooth boundary. Then all outer billiard trajectories of $K$ are planar if and only if $K$ is a ball up to an affine symplectomorphism.
\end{corollary}

\begin{remark}
The similar statement is false for \emph{symplectic billiards} which can also be defined for any convex body in the symplectic $\mathbb{R}^{2n}$ (see \cite{albtab2017} for the actual definition). In case of symplectic billiards there is no such convex body in $\mathbb{R}^{2n}$ with $n \ge 2$ which has all planar symplectic billiard trajectories. Indeed, for any two points on the boundary of such a body there is a symplectic billiard trajectory through them and the planarity assumption would imply that characteristic tangent lines at any two points intersect or parallel. This immediately implies that all characteristics lines lie in the same plane and contradicts $n \ge 2$. 
\end{remark}

\subsection*{Acknowledgments} The authors thank Sergei Tabachnikov for drawing our attention to these problems and useful discussions, Alexey Balitskiy, Mark Berezovik, Florent Balacheff, Viktor Ginzburg, Yaron Ostrover, Felix Schlenk, and the unknown referees for useful remarks and corrections.

\section{All characteristics closed with same action}
\label{section:all-closed}

In this section we study convex bodies in $\mathbb R^{2n}$ having all characteristics closed with the same action. These are related to Viterbo's conjecture \cite{vit2000}: \emph{For convex bodies $K$ in the symplectic $\mathbb R^{2n}$,}
\[
\vol K \ge \frac{c_{EHZ}(K)^n}{n!}.
\]
In view of the description of the Ekeland--Hofer--Zehnder capacity $c_{EHZ}(K)$, for a smooth convex body $K$, as the minimal action of a closed characteristic $\gamma$ on $\partial K$ (see \cite{aao2008,aao2012}), it makes sense to check the inequality in the case when all or almost all (in terms of measure) characteristics of $\partial K$ are closed and have the same action, as expressed in Conjecture \ref{conjecture:all-same}.

In \cite[Proposition~3.3]{apb2012} (going back to observations in \cite{boothbywang1958,weinstein1974}) the following was observed when discussing contact geometry approaches to the Viterbo conjecture. We provide the proof here, since the ideas from those proofs are needed in our main argument.

\begin{lemma}[\'Alvarez Paiva, Balacheff]
\label{lemma:apb-besse}
If $\lambda$ is a contact form on a closed manifold $M$ of dimension $2n-1$ and all its Reeb orbits are closed with the same period $c$ then $\int_M \lambda\wedge (d\lambda)^{n-1} = k c^n$, for a nonzero integer $k$.
\end{lemma}

\begin{proof}[Sketch of the proof]
Rescale $\lambda$ so that all Reeb orbits are of period $1$, note that the period equals the action in view of the equality $\lambda(R) \equiv 1$. There arises the free action of the circle $S=\mathbb R/\mathbb Z$ on $M$ and the quotient space $Y = M/S$.

The space $Y$ carries the symplectic structure of the symplectic reduction, $\omega_Y$, which is defined by lifting the tangent space of $Y$ to $M$ and applying $\omega=d\lambda$; this does not depend of the lift since $i_R\omega = L_R\omega= 0$ and $\omega$ is $S$-invariant. It is crucial to note that $\omega_Y$ is a symplectic form on $Y$ and it represents the Chern class of the circle bundle $M \to Y$. Hence $[\omega_Y]$ is an integral cohomology class and so is $[\omega_Y^{n-1}]$. Then we may perform the integration by a version of Fubini formula for the fibration, noting that $\int \lambda$ over each fiber is $1$, 
\[
\int_M \lambda\wedge \omega^{n-1} = \int_Y \omega_Y^{n-1} = k,
\]
for a positive integer $k$, since the symplectic volume of $Y$ is positive.
\end{proof}

In our setting we may put $M=\partial K$ and, assuming the origin in the interior of $K$, take as contact form on $M$ 
\[
\lambda = \frac{1}{2}\sum_i \left( p_idq_i - q_idp_i\right).
\]
Then $d\lambda = \omega$, $d (\lambda\wedge \omega^{n-1}) = \omega^n$ and by the Stokes formula we have:
\[
\int_M \lambda\wedge \omega^{n-1} = \int_K \omega^n = n! \vol K,
\]
and in view of appropriately scaled Lemma~\ref{lemma:apb-besse}, $\vol K = k c^n / n!$. In fact, it is possible to show that $k=1$ (as hinted by \cite[Remark~2]{weinstein1974} and done explicitly in the proof of Theorem~3.4 of the arxiv version~1 of \cite{apb2012}) and there is in fact the case of equality in Viterbo's inequality for bodies all of whose characteristics are closed with the same action. For completeness, we provide a proof here, which is different from the one given in \cite{apb2012}.

\begin{lemma}[\'Alvarez Paiva, Balacheff]
\label{lemma:same-action}
If $\lambda$ is a contact form on a sphere $M$ of dimension $2n-1$ and all its Reeb orbits are closed with the same period $c$ then $\int_M \lambda\wedge (d\lambda)^{n-1} = c^n$. As a consequence, for a convex $K\subset\mathbb R^{2n}$, if the boundary $\partial K$ is smooth and all its characteristics are closed with the same action $c$ then $\vol K = c^n/n!$.
\end{lemma}
\begin{proof}
Since the period is constant, one obtains a free circle $S$ action\footnote{We caution the reader that the word ``action'' here means that the group $S$ acts on the manifold $M$, not the symplectic action functional in previous occurrences of ``action''.} on $M$ and $Y=M/S$ is also a smooth manifold. 

In equivariant topology \cite{hsiang1975} it is standard that the cohomology of the quotient $Y = M/S$ is naturally equal to the equivariant cohomology
\[
H_{S}^*(M) = H^*( (M \times ES) / S).
\]
The latter space in the formula is the total space of a fibration over $BS$ ($=\mathbb CP^\infty$) with fiber $M$. Hence $H^*(M/S)$ is calculated through the spectral sequence with the second page equal to
\[
E_2^{x,y} = H^x(BS; H^y(M)).
\]
Since $M$ is a $(2n-1)$-dimensional sphere, it follows that the page $E_2$ is empty for $0<y<2n-1$ and the bottom row $H^x(BS; H^0(M))$ of this spectral sequence in the range $0\le x \le 2n-2$ survives until the page $E_\infty$. From the dimension considerations it follows that this part of the bottom row actually equals the resulting cohomology $H^*(Y)$.

Now we approximate $BS$ by $\mathbb CP^N$ with sufficiently large $N$. The above argument says that the classifying map $F: Y \to \mathbb CP^N$ (obtained from an $S$-equivariant map $M\to S^{2N+1}$) induces an isomorphism in integral cohomology in dimensions from $0$ to $2n-2$. From the meaning of the Chern class (as in the proof of Lemma \ref{lemma:apb-besse}) we obtain that $F^* [\omega_{FS}] = [\omega_Y]$, where $\omega_{FS}$ is the Fubini--Study form in $H^2(\mathbb CP^N)$ and $\omega_Y\in H^2(Y)$ is the reduced symplectic form. It also follows that $F^* [\omega_{FS}^{n-1}] = [\omega_Y^{n-1}]$.

To integrate $\omega_Y^{n-1}$ over $Y$ is the same as to pair $\omega_{FS}^{n-1}$ with the homology class, represented by $F(Y)$ as a homology $(2n-2)$-dimensional cycle. But the cohomology isomorphism implies the homology isomorphism up to dimension $2n-2$, hence $F(Y)$ is homologous to $\mathbb CP^{n-1}\subset \mathbb CP^N$, thus showing that $[\omega_Y^{n-1}]$ must be the generator of $H^{2n-2}(Y)$ and $\int_Y \omega_Y^{n-1} = \pm 1$, where we choose $+1$ because we choose the orientation of $Y$ in accordance with its symplectic form.

Passing to a convex $K\subset\mathbb R^{2n}$, we set $M=\partial K$. After a translation we may assume that the origin is in the interior of $K$ and $M$ is a level set $\{H=1\}$ of a $2$-homogeneous Hamiltonian. Then the action of a closed characteristic equals its Hamiltonian period, which finishes the proof. 
\end{proof}

Till the end of this section we discuss some open questions about the situation when almost all but not all characteristics are closed with the same action. The general idea is that such a behavior may imply the Viterbo equality. But we will see that the situation is not that simple. By ``almost all'' we mean that the characteristics of the given same action cover the boundary of $K\subset\mathbb R^{2n}$ up to $(2n-1)$-dimensional measure zero.

Relating the action of the characteristics to the period of the Reeb flow, as in the above proof, we observe that when almost all the characteristics have the same period $c$ then by continuity $c$ is the period of any characteristic, possibly not minimal. If the period $c$ is not minimal for a characteristic $\gamma$ then its period is $c/k$ for an integer $k\ge 2$, hence the period (action) may drop below $c$ for a set of characteristics of zero measure. Assuming no action (period) drop, we just return to the case of all (without exceptions) characteristics closed with same action. 

\begin{example}
If for some bad characteristic the action drops below the common value $c$, then the capacity $c_{EHZ}$ drops and we cannot expect the equality $\vol K = c^n/n!$. The body
\[
K = \{p_1^2  + q_1^2 + 2p_2^2 + 2q_2^2 = 1\}
\]
has most of the characteristics closed with action $\pi$, while one of the characteristics is closed with action $\pi/2$, the volume being $\pi^2/4$. 
\end{example}

\begin{question}
\label{question:smooth-almost-all}
Is it possible to prove the Viterbo inequality for smooth convex bodies under the assumption that almost all the characteristics are closed with the same action, but there are action drops for some exceptional characteristics?
\end{question}

\begin{remark}[Pointed out by an unknown referee]
The discussion in \cite[The Theorem and Section~2]{wadsley1975} may also be relevant to this question. In \cite[Theorem~1.1]{mazzurade2023} it is stated that in dimension $3$ every contact sphere with all Reeb orbits periodic is strictly contactomorphic to an ellipsoid. This implies that the above question has a positive answer in dimension $4$ by simply inspecting the case of ellipsoids. The case of higher dimensions in Question~\ref{question:smooth-almost-all} seems open so far.
\end{remark}

In the case, when $K$ is not smooth, and not all the characteristics are well-defined, one may also guess from the results of \cite{aao2012,balitskiy2016} that in the case when almost all characteristics are closed and have the same action $c = c_{EHZ}(K)$ (so that there is no action drop), then $\vol K$ must be $c^n/n!$. But this expectation also fails.

\begin{example} Consider the body
\[
K = \{p_1^2 + q_1^2 \le 1, p_2^2 + q_2^2 \le 1\};
\]
almost all characteristics of $\partial K$ are either trajectories of the Hamiltonian $p_1^2 + q_1^2$, or trajectories of $p_2^2 + q_2^2$. In both cases they have action $\pi$, and the remaining part of $\partial K$ (where both inequalities become equalities) may be described as spanned by the generalized characteristics
\[
p_1 = \cos t,\ q_1 = \sin t,\ p_2 = \cos (t+\phi),\ q_2 = \sin (t+\phi)
\]
for $\phi\in [0,2\pi)$, all of which have action $2\pi$. It is even possible to avoid considering generalized characteristics, noting that the inclusion of the unit ball $B^4\subset K$ together with the monotonicity of $c_{EHZ}$ implies $c_{EHZ}(K)\ge c_{EHZ}(B^4) = \pi$. Hence, in this case $c_{EHZ}(K)=\pi$ and
\[
\vol K = \pi^2 > c_{EHZ}(K)^2/2 = \pi^2/2.
\]
\end{example}

\section{All characteristics are planar}
\label{section:all-planar}

In this section we present the proof of Theorem \ref{theorem:all-planar} split into several lemmas. 

\begin{remark}
Let us mention that characterizing ellipsoids is an old business. For example, Blaschke's characterization of ellipsoids \cite{bla1916,gruber1997} asserts the following: \emph{A convex body $K$ in dimension at least $3$ is an ellipsoid if and only if all its two-dimensional projections have the property that the preimage $($in $K)$ of the boundary of the image of $K$ is planar.} 

In our situation we have something similar. If a closed characteristic $\gamma\subset \partial K\subset\mathbb R^{2n}$ is planar then all its velocities lie in a two-dimensional plane. Then all the tangent to $K$ hyperplanes at points of $\gamma$ have normals in a two-dimensional plane. Hence the projection of $K$ to this plane of normals has the Blaschke property. The bad thing is that we only have a $(2n-2)$-dimensional family of such projections ($2n-2$ is the dimension of the space of all characteristics), while Blaschke's characterization needs the property for all projections, whose space is the Grassmannian $G_{2,2n}$ of dimension $4n-4$.
\end{remark}

Since, by the above remark, the classical theorem does not solve our problem, we go down to our own argument. 

Obviously, if $K$ is a ball then all characteristics are planar and the planarity of characteristics is preserved under an affine symplectomorphism. Hence it remains to prove the opposite direction. All lemmas in this section are stated and proved under the hypotheses of Theorem \ref{theorem:all-planar}, that is, $K\subset\mathbb R^{2n}$ is a strongly convex body with smooth boundary and $n\ge 2$.

\begin{lemma}
\label{lemma:planar same}
If all characteristics of $K$ are planar then they all have the same action.
\end{lemma}
\begin{proof}
Since the characteristics are planar, they are all closed. From the smooth dependence of a characteristic on the starting point one sees that the $2$-planes spanning the characteristics depend smoothly on an initial point on the characteristic. Moreover, when a characteristic is deformed with non-zero derivative then its containing $2$-plane is also deformed with non-zero derivative. It thus follows that the $2$-planes of all characteristics of $K$ constitute a smooth submanifold $Y$ of the affine Grassmannian $AG_{2,2n}$. This manifold $Y$ naturally parametrizes the characteristics of $K$ and is connected, since $\partial K$ is itself connected.

Note that, for the action functional $\int_\gamma \lambda$ (for a primitive $d\lambda = \omega$), every closed characteristic $\gamma\subset\partial K$ is a critical point since for a vector field $V$ along $\gamma$ tangent to $\partial K$ one has
\[
L_V \left(\int_\gamma \lambda\right) = \int_\gamma L_V\lambda = \int_\gamma i_V d\lambda + d i_V\lambda = \int_\gamma i_V \omega + \int_\gamma d\lambda(V) = \int_\gamma \omega(V,\dot \gamma)\;dt = \int_\gamma 0\;dt = 0.
\] 
Therefore the value of the action is locally constant on $Y$. Since $Y$ is connected, this value of the action is constant on the whole $Y$.
\end{proof}

All characteristics on $\partial K$ are closed with the same action by Lemma~\ref{lemma:planar same} and therefore Lemma~\ref{lemma:same-action} applies, resulting in the following lemma.

\begin{lemma}
\label{lemma:planar viterbo}
If all characteristics on the boundary of $K$ are planar then we have the Viterbo equality:
\[
\vol K = \frac{c(K)^n}{n!}.
\]
\end{lemma}

\begin{lemma}
\label{lemma:planar to ellipses}
If all characteristics of $K$ are planar then they are all ellipses.
\end{lemma}
\begin{proof}
Having the same period for planar characteristics, we can again consider the circle $S$ action on $M = \partial K$ and the circle bundle $\eta: M \to M/S = Y$. For this it is useful to translate $K$ so that it contains the origin and assume that $\partial K$ is the level set $\{H=1\}$ of a $2$-homogeneous Hamiltonian. Then the period of the Hamiltonian flow will be the same as the action of a closed characteristic.

Define a connection on the bundle $\eta : M \to Y$ as a splitting of its tangent space to vertical and horizontal subspaces: For each point $p \in M$ on a characteristic $\gamma$ lying in the plane $L$ define 
\[
\Ver_p = \ker \omega |_{T_pM}
\]
\[
\Hor_p =  JL^{\perp} = L^{\perp_\omega}.
\]
Note that the $\omega$-orthogonal $L^{\perp_\omega} = \{y \in \mathbb R^{2n} \ |\ \omega(x, y) = 0, \forall x \in L\}$ is a subspace of $T_pM \subset T_p\mathbb R^{2n}$, since $T_p M$ is $\omega$-orthogonal to $\ker \omega|_{T_pM}$ (spanned by the tangent of $\gamma$ by the definition of a characteristic), $\ker \omega|_{T_pM}\subset L$, and taking the $\omega$-orthogonal reverses the inclusions.

Since the characteristic $\gamma$ has nonzero action and lies in $L$, $L$ is symplectic and is not contained in $L^{\perp_\omega}\subset T_pM$. Hence $L \cap L^{\perp_\omega} = \ker \omega|_L$ is not equal to $L$ and may only be the zero subspace of the two-dimensional $L$. This means that $L$ and $L^{\perp_\omega}$ are symplectic subspaces of $\mathbb R^{2n}$. Notice that $\Hor_p = L^{\perp_\omega}$ is also a connection on the $S$-bundle $\eta : M \to Y$ because it is preserved under the Hamiltonian flow on $M$ which defines the circle $S$ action on $M$.

Consider also the affine plane bundle $E\to Y$ consisting of the planes $L$ spanned by the characteristics. Take the same horizontal subspaces $L^{\perp_\omega}$ to define a connection on $E$. The thus obtained connection is actually affine (its parallel transport maps are affine transformations) since its horizontal planes in different points of $L$ are obtained from each other by a translation. Hence the holonomy groups of $E\to Y$ consist of some affine transformations of the fiber planes $L$. Going back to our $S$-bundle $M\subset E$ we have that its holonomy group also consists of affine transformations of fibers (inside their planes).

The circle bundle $\eta: M \to Y$ has a curvature form, which represents its first Chern class $c_1(\eta)\in H^2(Y)$ (the class does not depend on a connection). The bundle is not trivial and hence $c_1(\eta)$ (the only invariant of circle bundles) must be nonzero. The group $H^2(Y) = \mathbb Z$ has no torsion (see its description in the proof of Lemma~\ref{lemma:same-action}) and therefore $c_1(\eta)$ has no torsion and does not vanish in the de Rham cohomology (with $\mathbb R$ coefficients). Hence the curvature form of $\eta$ cannot be everywhere zero and the holonomy group of $\eta$ is not discrete.

Now consider the John ellipse (the ellipse of the maximal area) $E(L)$ in the planar convex body $K \cap L$ inscribed in a characteristic $\gamma = \partial K \cap L$. Since the holonomy group preserves the body $K \cap L$ it also preserves the John ellipse $E(L)$ (by its uniqueness). Then in a coordinate system in which $E(L)$ is a circle we have that the holonomy is a non-discrete subgroup of $\mathrm{SO}(2)$ (it is connected since $M$ is connected) and hence the holonomy group equals the whole group $\mathrm{SO}(2)$. Therefore the section $K \cap L$ is preserved under any such rotation and then it must coincide with its John ellipse. This completes the proof.
\end{proof} 

\begin{lemma}
\label{lemma: central symmetry}
If all characteristics on the boundary of $K$ are planar and at least one of them is centrally symmetric then $K$ is centrally symmetric.
\end{lemma}
\begin{proof}
In order to prove the central symmetry of $K$, we consider the Minkowski sum of $K$ and $-K$
\[
K - K = \{x - y \in \mathbb R^{2n}\ |\ x, y \in K\}.
\]
Our strong convexity assumptions are equivalent to the fact that the support function $h_K(u)$ of $K$ is twice differentiable outside the origin and its Hessian is positive definite on directions orthogonal to $u$. The support function of $K-K$ is $h_K(u)+h_K(-u)$ and it obviously inherits these properties. Hence we do not leave our class of convex bodies when doing the Minkowski sum.

Let us first show that if all characteristics of $K$ are planar then all characteristics of $K - K$ are planar too. Let $\gamma$ be a characteristic of $K$ lying in the affine plane $L$. Denote the plane $L$ translated to the origin by $L_0$ and note that the Minkowski sum $L-L$ equals $L_0$.

All normal vectors at the points of $\gamma$ lie in the plane $L' = JL_0$. Consider the orthogonal projection $P : \mathbb R^{2n}\to L'$. The boundary points of the image $P(K)$ are exactly the images of the points of $\partial K$ whose normal vectors are parallel to the plane $L'$. By the strict convexity of $K$, the latter point set is a closed curve containing $\gamma$ and in fact equal to $\gamma$. Hence $P(\gamma) = \partial P(K)$. 

Observe that for any convex body $B$ in dimension at least $2$ we have the equality for Minkowski sums $B-B=\partial B-\partial B$. Consider the body $K - K$ and the curve $\beta$ that bounds the two-dimensional planar surface $K\cap L- K\cap L$, which by the above observation equals $\gamma - \gamma\subset L_0\cap (K-K)$. Since the Minkowski sum commutes with linear projections, we have
\[
P(\gamma - \gamma) = P(\gamma) - P(\gamma) = \partial P(K) - \partial P(K) = P(K)-P(K) = P(K - K).
\]
Hence the curve $\beta = \partial (\gamma-\gamma)$ is contained in the boundary of $K-K$ and the normal to $K-K$ in any point of $\beta$ is contained in $L'$. Therefore the characteristic directions of $\partial (K-K)$ at all points of $\beta$ are in $JL' = JJL_0 = - L_0 = L_0$. Hence $\beta$ is a characteristic of $K - K$.

Let $\gamma$ be a centrally symmetric characteristic of $K$ existing by the hypothesis of the lemma. Let $\gamma$ have action $c$, which is the same as the action of any other characteristic and equal to $c_{EHZ}(K)$. Without loss of generality let the center of symmetry of $\gamma$ be the origin. Then $\beta = \partial (\gamma - \gamma)  = 2\gamma$ is a characteristic of $K - K$ with action $4c$. By the Viterbo equality of Lemma~\ref{lemma:planar viterbo}, $\vol (K - K) = (4c)^n/n! = 4^n \cdot c^n/n! = 2^{2n} \vol K$. Hence in the Brunn--Minkowski inequality for the bodies $K$ and $-K$ we have equality
\[
(\vol (K - K))^{1/2n} = 2(\vol (K))^{1/2n} = (\vol K)^{1/2n} + (\vol (-K))^{1/2n}.
\]
From the description of the equality case of the Brunn--Minkowski inequality (see \cite[Theorem~8.1]{gruber2007}) it follows that $K$ and $-K$ are homothetic and thus $K$ is centrally symmetric.
\end{proof}
 
\begin{lemma}
\label{lemma:ellipses to ellipsoid}
If all characteristics on the boundary of $K$ are ellipses then $K$ is an ellipsoid. 
\end{lemma}
\begin{proof}
From Lemma \ref{lemma: central symmetry} we know that $K$ has a center of symmetry, let it be the origin. Let us show that the center of $K$ belongs to all characteristic planes. Indeed, for each point $x$ there is a point $y$ on the same characteristic with parallel velocity with opposite direction. Hence the normal vectors to $\partial K$ at $x$ and $y$ are parallel with opposite direction. From central symmetry of $K$, the line through two points with opposite normal vectors is going through the center. Thus the line through $x$ and $y$ is passing through the center and also lies in the plane of their characteristic, so the center belongs to each characteristic's plane and it is also the center of any characteristic.

Define the symplectic polar body 
\[
K^\omega = \{y\in\mathbb R^{2n}\ |\ \omega(x,y)\le 1,  \forall x\in K\}.
\]
Note that $K^\omega$ is the $J$-rotated polar body defined by the inner product on $\mathbb R^{2n}$,
$K^\omega = JK^\circ$, where 
\[
K^\circ = \{y\in\mathbb R^{2n}\ |\ x\cdot y \le 1,  \forall x\in K\}.
\]
Since $K$ is strictly convex with positive definite second fundamental form at each boundary point, has smooth boundary, and is centrally symmetric, $K^\omega$ is also strictly convex with positive curvature, has smooth boundary and is centrally symmetric. Let us explain some details that we need in the rest of the argument. Since $K$ is strictly convex, for any point $x \in \partial K$ there exists a unique vector $y(x) \in K^\omega$ such that
\[
\omega(x,y(x)) = 1.
\]
Note that from the positivity of the second fundamental form of $\partial K$ it follows that the dependence of $y(x)$ on $x\in \partial K$ is smooth and its inverse is also smooth, since this is just a $J$-rotated dependence of the normal on the point of $\partial K$. This also implies the smoothness of $\partial K^\omega$, since its normal $Jx(y)$ therefore has smooth dependence on the point $y\in \partial K^\omega$.

Let us show that all characteristics of $K^\omega$ are planar and lie in the same planes as those of $K$. Consider any characteristic $\gamma\subset\partial K$ in the $2$-plane $L$; we show that there is a corresponding characteristic of $\partial K^\omega$ in the same plane. The vector $y(x)$ gives the characteristic direction at the point $x \in \partial K$ (it is a $J$-rotated normal vector at the point $x$) and therefore lies in the plane $L$. Hence $\gamma^\omega = \{y(x)\ |\ x\in \gamma\}$ is a smooth curve on $\partial K^\omega$ which lies in $L$. 

Now, interchanging $K$ and $K^\omega$ and noting that the polar of $K^\omega$ is $(K^\omega)^\omega = - K = K$, we get that for any $y \in \gamma^\omega$ its corresponding $x \in \gamma$ gives a characteristic direction of $\partial K^\omega$ at the point $y$ and lies in the plane $L$. Hence $\gamma^\omega$ is a characteristic of $\partial K^\omega$ and, as was shown above, lies in the plane $L$.

Consider the vector $x\in \partial K$ of the smallest length. The normal to $\partial K$ at $x$ and $-x$ (from the central symmetry) is parallel to $x$. Let $\gamma$ be the characteristic through the point $x$ that lies in the plane $L$. Then the normal vector $N(x)$ at $x$ (which is parallel to $x$) and $JN(x)$ both lie in $L$, $L \cap JL \not= \emptyset$. Since $L\cap JL$ is a $J$-invariant linear subspace of $\mathbb R^{2n}$ then $L = JL$ and $L$ is $J$-invariant itself. Then the action of the characteristic $\gamma$ is equal to the area of $K \cap L$ and equal to $c_{EHZ}(K)$ (because all the characteristics have the same action).

Since $\gamma$ and $\gamma^\omega$ are polar ellipses the product of their areas is equal to $\pi^2$ by the Blaschke--Santalo inequality (which is an equality for an ellipse). By the Viterbo equality of Lemma~\ref{lemma:planar viterbo} for both $K$ and $K^\omega$ ($K$ and $K^\omega$ are both strongly convex with smooth boundary) we obtain
\[
\vol(K) \cdot \vol(K^\circ) = \vol(K) \cdot \vol(K^\omega) = \frac{c_{EHZ}(K)^n}{n!} \cdot  \frac{c_{EHZ}(K^\omega)^n}{n!} = \left(\frac{\pi^n}{n!}\right)^2.
\]
Hence we have the equality in the Blaschke--Santalo inequality, which means (see \cite{saintraymond1981,petty1985}) that $K$ is an ellipsoid.
\end{proof}

Combining Lemmas \ref{lemma:planar to ellipses}, \ref{lemma: central symmetry}, \ref{lemma:ellipses to ellipsoid} we get that if all characteristics of $K$ are planar then $K$ is an ellipsoid. The Hamiltonian that defines $K$ is then a positive definite quadratic form. When this quadratic form is transformed to its canonical form
\[
Q = \sum_{i=1}^n a_i (p_i^2 + q_i^2)
\]
after a linear symplectic transformation, the coefficients $a_i$ must be all equal, since all the periods of the characteristics must be the same. Hence in these canonical coordinates $K$ is a ball.

\section{All outer billiard trajectories are planar}
\label{section:outer}

The proof of Corollary~\ref{corollary:outer-planar} follows from Theorem~\ref{theorem:all-planar} and the following lemma.

\begin{lemma}
\label{lemma:outer-planar}
All outer billiard trajectories of a smooth strictly convex body $K\subset\mathbb R^{2n}$ are planar if and only if all characteristics of $K$ are planar.
\end{lemma}

\begin{proof}
If all characteristics of $K$ are planar then it is easy to extend any outer billiard trajectory in the $2$-plane of the characteristic to which either of the segment of the trajectory is tangent. From the uniqueness theorem for the outer billiard trajectories (see Lemma~13.2 in \cite{tabachnikov1993}) it follows that all the outer billiard trajectories of $K$ are planar.

Let us now prove that if all outer billiard trajectories of $K$ are planar then all characteristics of $K$ are planar. Equivalently, we will prove that any point $z \in \partial K$ lies on a planar characteristic.

Let $\ell$ be the tangent line at $z$ with a characteristic direction and $L$ be a plane with some outer billiard trajectory passing through $z$ and starting at some point $z_1 \in \ell$. Denote $\gamma = L \cap \partial K$. Call a point $x \in \gamma$ \emph{good} if a characteristic direction at $x$ lies in $L$. Since the outer billiard trajectory starting at $z_1$ touches $\gamma$ at good points the set of good points has at least three points and between any two points of $\gamma$ with parallel characteristic directions there is a good point.

If the set of good points is a dense subset of $\gamma$ then from the continuity of characteristic directions all points of $\gamma$ are good and $\gamma$ is by definition a planar characteristic passing through $z$.

Suppose that the set of good points is not dense. Then there exist at least two good points $x_1$ and $x_2$ of $\gamma$ such that there is no good points between them on one of two segments of curve $\gamma$ (the segment is on the same side of the line $x_1x_2$ as the intersection of characteristic directions at the points $x_1$ and $x_2$). 

Since any outer billiard trajectory is planar, the outer billiard trajectory touching $K$ at the points $x_1$ and $x_2$ (uniquely determined by \cite{tabachnikov1993} and probably different from the trajectory starting at $z_1$) lies in the plane $L$. Denote its consecutive points of tangency with $\gamma$ by $x_1, x_2, x_3, \dots$. 

If there is no good point of $\gamma$ between $x_1$ and $x_2$ then there is no good point between $x_2$ and $x_3$. Otherwise, one could consider another trajectory through a good point $y$ (lying between $x_2$ and $x_3$) and $x_2$, which should have a point of tangency to $\gamma$ (and hence a good point of $\gamma$) between $x_2$ and $x_1$ due to planarity of trajectories and the definition of outer billiard reflection (the point of tangency is always the middle point between the consecutive vertices of the trajectory). This implies that there is no good point of $\gamma$ between $x_3$ and $x_4$ and so on.

We continue stepping by good points of $\gamma$. Let us show that the trajectory must at some stage return to $x_1$, so it can not stop by getting closer and closer to the curve and the number of good points on $\gamma$ is in fact finite. Notice that for inner billiard there are such bodies for which some trajectories can approach the boundary of the body (see \cite{halpern1977}).

If the trajectory does not return to $x_1$ then there exists a limit point 
\[
x_{\infty} = \lim_{m \to \infty} x_m,
\]
which is also a good point (by continuity). Choosing a point $x_m$ sufficiently close to $x_{\infty}$ we can find an outer billiard trajectory passing consecutively through points $x_k$, $x_m$ and $x_{\infty}$ for some $k > 1$ since any point of tangency is a good point and the sequence $\{x_n\}$ lists all good points on the segment from $x_1$ to $x_\infty$. Then between $x_{\infty}$ and $x_n$ there is an infinite number of good points, while between $x_k$ and $x_m$ there is only a finite number of them. This cannot be true by the similar argument as before: For any point $x_i$ between $x_{\infty}$ and $x_m$, one can consider a trajectory through $x_i$ and $x_m$. This trajectory should pass through a good point $x_j$ between $x_m$ and $x_k$ (again since any point of tangency is a good point and the sequence $\{x_n\}$ lists all good points on the segment), and all these points $x_j$ should be different for different $x_i$. This all means that in fact there are a finite number of good points on $\gamma$.   

This implies that the outer billiard trajectory passing through $z_1$ and $z$, the one we started with, is also periodic. Going to the start of our argument, we conclude that if there is a non-periodic outer billiard trajectory through point $z$ then $z$ lies on a planar characteristic.

Assume now that all outer billiard trajectories through $z$ are periodic. Parameterize the initial point (lying on $\ell$) of outer billiard trajectory by $z_1(t) = z - tv$ where $v$ is a unit vector in the characteristic direction at $z$ and $t \in (0, \infty)$. Let the trajectory starting at the point $z_1(t)$ lie in the plane $L(t)$. Denote vertices of the trajectory starting at $z_1(t)$ by $z_1(t), z_2(t), z_3(t),...$ and a tangency point between $z_i(t)$ and $z_{i+1}(t)$ by $x_i(t)$. Since by \cite{tabachnikov1993} an outer billiard map is a symplectomorphism of  $\mathbb R^{2n} \setminus K$, all $z_i(t)$ and $x_i(t) = (z_i(t) + z_{i+1}(t))/2$ are continuous functions of $t$.

We use the following \emph{uniform distribution of good points between the tangency points}.

\begin{claim}
\label{claim:uniform}
If the outer billiard trajectory in the plane $L(t)$ does not show that the curve $\gamma(t)=L(t)\cap\partial K$ is a characteristic of $K$ then the number of good points of $\gamma(t)$ between the consecutive points of tangency $x_i(t)$ and $x_{i+1}(t)$ of the trajectory is independent of $i$.
\end{claim}

\begin{proof}
The proof is already given above: The number of good points on $\gamma(t)$ is finite, there is an injection of the set of good points between $x_{i-1}(t)$ and $x_i(t)$ to the set of good points between $x_i(t)$ and $x_{i+1}(t)$, and vice versa.
\end{proof}

Now consider the sets
\[
A_n = \left\{ t  \in (0, \infty) \middle| \text{ outer billiard trajectory starting at } z_1(t) \text{ has period } n\right\}.
\] 
We are going to show that those sets are closed in $(0, \infty)$. Suppose $\{t_m\}_{m \in \mathbb N} \subset A_n$ and 
\[
\lim_{m \to \infty} t_m = t' \in (0, \infty).
\]
Then $z_{n+1}(t_m) = z_1(t_m)$ for any $m$  and
 \[
 z_{n+1}(t') = \lim_{m \to \infty} z_{n + 1}(t_m) = \lim_{m \to \infty} z_1(t_m) = z_1(t').
 \]
Generally, the trajectory starting at $z_1(t')$ has a period $k$ which divides $n$. We assume that $k < n$ and deduce a contradiction. 

For any $\varepsilon>0$, the inequality $|x_i(t_m) - x_i(t')| < \varepsilon$ holds for sufficiently large $m$ and any $i$. We can take $\varepsilon$ so small that the $\varepsilon$-neighborhoods of $x_i(t')$ do not overlap. Then the curve $\gamma(t_m)$ has $k$ disjoint intersections with those neighborhoods and the tangency points of the trajectory belong to those intersections.

Consider a trajectory starting at $z_1(t_m)$, then $x_{k+1}(t_m)$ lies in the $\epsilon$-neighborhood of $x_1(t') = z = x_1(t_m)$ for sufficiently large $m$. There are three cases: $x_{k+1}(t_m)$ lies after $x_1(t_m)$, before $x_1(t_m)$ or coincide with $x_1(t_m)$. 

From Claim~\ref{claim:uniform} it follows that $x_{k+1}(t_m) \not = x_1(t_m)$ if $k < n$. Indeed, if $x_{k+1}(t_m) = x_1(t_m)$ then $x_k(t_m)\neq x_n(t_m)$ since the period is $n$, not $k$. Since the numbers of good points between $x_k(t_m)$ and $x_{k+1}(t_m) = x_1(t_m)$ and between $x_n(t_m)$ and $x_1(t_m)$ are equal, we must have $x_k(t_m)= x_n(t_m)$, a contradiction.

Assume now that $x_{k+1}(t_m)$ lies after $x_1(t_m)$, then $x_{k+2}(t_m)$ lies after $x_2(t_m)$ because of Claim~\ref{claim:uniform}. Continuing the argument, we obtain that $x_{n + 1}(t_m)$ lies after $x_1(t_m)$ and still in the $\epsilon$-neighborhood of $x_1(t')$, so $x_{n + 1}(t_m)$ cannot equal $x_1(t_m)$, a contradiction. 

For the case when $x_{k+1}(t_m)$ lies before $x_1(t_m)$, the same argument works. Hence, we get a contradiction in any case and conclude that $k = n$ and $t'\in A_n$. Hence $A_n$ is a closed set and since the connected interval is the disjoint union
\[
(0,+\infty) = \bigsqcup_{n=3}^\infty A_n,
\]
then in fact all the trajectories tangent to $\partial K$ at $z$ have the same period. But this cannot be the case, since the period must be arbitrarily large if we start an outer billiard trajectory from a point $z_1$ very close to $z$. Indeed, since $K$ is strictly convex, one shows by induction on $n$ that for any $n$, $z_n(t)\to z$ and $z_n(t)\neq z$ when $t\to +0$ and $z_1(t)\to z$.

\end{proof}

\bibliography{../Bib/karasev}
\bibliographystyle{abbrv}
\end{document}